\documentclass[preprint,12pt]{elsarticle}

\usepackage[T1]{fontenc}
\usepackage[utf8]{inputenc}
\usepackage{amsmath,amssymb,amsthm,mathtools}
\usepackage{microtype}
\usepackage{graphicx}
\usepackage{booktabs}
\usepackage{hyperref}
\hypersetup{hidelinks}

\journal{Applied and Computational Harmonic Analysis}

\newtheorem{theorem}{Theorem}[section]
\newtheorem{corollary}[theorem]{Corollary}
\newtheorem{proposition}[theorem]{Proposition}
\newtheorem{lemma}[theorem]{Lemma}
\newtheorem{remark}[theorem]{Remark}

\DeclareMathOperator{\rank}{rank}
\DeclareMathOperator{\dist}{dist}
\newcommand{\R}{\mathbb{R}}
\newcommand{\C}{\mathbb{C}}
\newcommand{\F}{\mathbb{F}}
\newcommand{\PP}{\mathbb{P}}
\newcommand{\EE}{\mathbb{E}}
\newcommand{\cS}{\mathcal{S}}

\begin{document}

\begin{frontmatter}

\title{Extreme least singular values of Gaussian row submatrices and a phase retrieval stability problem}

\author[huji]{Yitzchak Shmalo}
\ead{yitzchak.shmalo@gmail.com}
\address[huji]{Einstein Institute of Mathematics, The Hebrew University of Jerusalem, Givat Ram, Jerusalem, Israel}

\begin{abstract}
Let $\F$ be either $\R$ or $\C$, and put $d_\F=\dim_\R\F$. Let $A_m\in\F^{N_m\times m}$ have independent
standard Gaussian entries and assume $N_m/m\to\gamma>1$. We determine the exponential scale of the
smallest least singular value among all square row submatrices:
\[
  M_m^\F=\min_{T\subset [N_m],\, |T|=m}\sigma_{\min}(A_{m,T}).
\]
With $c_\gamma=\gamma^\gamma/(\gamma-1)^{\gamma-1}$ we prove
$\tfrac1m\log M_m^\F\xrightarrow{P}-d_\F^{-1}\log c_\gamma$, and when $N_m=\gamma m+O(1)$ the convergence
holds at rate $O(m^{-1})$. As an application, for real
phase retrieval at the critical number $N=2m-1$ of measurements the Balan--Wang stability parameter
satisfies $\omega(A_m)=4^{-m+o_P(m)}$, so the sharp exponential base in the Gaussian Balan--Wang problem
is $1/4$. The lower estimate extends to all bounded--density ensembles, and continuity is necessary.
\end{abstract}

\begin{keyword}
phase retrieval \sep random matrices \sep Gaussian matrices \sep least singular value \sep hard edge \sep complement property \sep second moment method
\MSC[2020] 60B20 \sep 42C15 \sep 94A12 \sep 15B52 \sep 15A18
\end{keyword}

\end{frontmatter}

\section{Introduction}

The least singular value of a random matrix is a basic hard-edge observable in random matrix theory.
This paper studies a related extremal quantity: the minimum of the least singular value over all
square row submatrices of a tall Gaussian matrix. The number of such submatrices is exponential in
the dimension, and the main point is to show that their strong dependence does not change the
exponential scale predicted by the one-matrix hard-edge small-ball exponent.

Let $\F\in\{\R,\C\}$ and let $d_\F=\dim_\R\F$, so that $d_\R=1$ and $d_\C=2$. A standard Gaussian
variable over $\R$ is $N(0,1)$; a standard Gaussian variable over $\C$ is $(\xi+i\eta)/\sqrt2$, where
$\xi,\eta$ are independent $N(0,1)$ variables. For $A\in\F^{N\times m}$ and $T\subset[N]$, $A_T$
denotes the submatrix formed from the rows indexed by $T$. Define
\begin{equation}\label{eq:def-M}
  M_m^\F(A)=\min_{T\subset[N],\ |T|=m}\sigma_{\min}(A_T).
\end{equation}

Our main result identifies the limiting exponential base for \eqref{eq:def-M} for every fixed row
aspect ratio $\gamma>1$.

\begin{theorem}[Extremal square submatrices]\label{thm:general}
Let $\F\in\{\R,\C\}$ and $d_\F=\dim_\R\F$. Let $A_m\in\F^{N_m\times m}$ have independent standard
Gaussian entries, and assume $N_m/m\to\gamma>1$. Put
\begin{equation}\label{eq:hgamma}
  h(\gamma)=\gamma\log\gamma-(\gamma-1)\log(\gamma-1),
  \qquad c_\gamma=e^{h(\gamma)}=\frac{\gamma^\gamma}{(\gamma-1)^{\gamma-1}}.
\end{equation}
Then
\begin{equation}\label{eq:main-limit}
  \frac{1}{m}\log M_m^\F(A_m) \xrightarrow{P} -\frac{h(\gamma)}{d_\F}.
\end{equation}
Equivalently,
\begin{equation}\label{eq:base}
  M_m^\F(A_m)=c_\gamma^{-m/d_\F+o_P(m)}.
\end{equation}
If in addition $N_m=\gamma m+O(1)$, for instance $N_m=\lfloor\gamma m\rfloor$, then for every fixed
$\varepsilon>0$,
\begin{equation}\label{eq:rate}
  \PP\left\{\left|\frac{1}{m}\log M_m^\F(A_m)+\frac{h(\gamma)}{d_\F}\right|>\varepsilon\right\}
  =O_{\gamma,\varepsilon,\F}(m^{-1}).
\end{equation}
\end{theorem}

The constant $c_\gamma$ is the exponential growth rate of the relevant binomial coefficient:
$\binom{N_m}{m}=\exp\{h(\gamma)m+o(m)\}$. For one $m\times m$ Gaussian matrix over $\F$, the
probability that $\sigma_{\min}$ is at most $t$ is of order $t^{d_\F}$ at the origin, up to powers of
$m$. If the square submatrices behaved independently, the minimum over $\exp\{h(\gamma)m+o(m)\}$ of
them would therefore be expected at scale $\exp\{-h(\gamma)m/d_\F+o(m)\}$. Theorem \ref{thm:general}
proves that this heuristic gives the correct exponential scale.

A closely related question was studied by Rademacher and Shu \cite{RademacherShu2022} in the smoothed
analysis of Frank--Wolfe methods. In our row notation, their results show that, when the number of
Gaussian rows is a fixed factor larger than the dimension, some square submatrix has exponentially
small least singular value, and all square submatrices have least singular value bounded below at a
possibly different exponential scale. The present paper refines that picture by identifying the exact
exponential rate, by treating both real and complex Gaussian matrices, and by extracting the sharp
Gaussian base in the Balan--Wang phase retrieval stability question.

The motivating application is the stability of real finite-dimensional phase retrieval; this is the
setting of Randomstrasse Open Problem 21. If $A\in\R^{N\times m}$ has rows $a_1,\ldots,a_N$, the real
phaseless measurement map is $x\bmod\{\pm1\}\mapsto |Ax|=(|\langle a_i,x\rangle|)_{i=1}^N$. The
finite-dimensional frame formulation goes back to Balan, Casazza and Edidin
\cite{BalanCasazzaEdidin2006}; see also \cite{BandeiraCahillMixonNelson2014,
ConcaEdidinHeringVinzant2015,GrohsKoppensteinerRathmair2020} for injectivity and stability background.
In the real case, injectivity is characterized by the complement property: for every partition of the
rows, one side must span $\R^m$. Thus the sharp generic threshold is $N=2m-1$.

Balan and Wang introduced quantitative stability parameters measuring the conditioning of spanning
subfamilies and related these parameters to Lipschitz bounds for phaseless inversion
\cite{BalanWang2015}; see also \cite{Balan2017,BalanZou2016,EldarMendelson2014}. At the real critical
threshold, for a full-spark $A\in\R^{(2m-1)\times m}$, define
\begin{equation}\label{eq:omega}
  \omega(A)=\min_{J\subset[2m-1]:\,\rank(A_J)<m}\sigma_m(A_{J^c}).
\end{equation}
Balan and Wang conjectured an exponential upper bound
\begin{equation}\label{eq:bw-bound}
  \omega(A)\le C\max_i\|a_i\|\, b^m
\end{equation}
with universal constants $C>0$ and $0<b<1$. Randomstrasse Open Problem 21 asks for the value of
$\omega(A)$ for iid Gaussian $A$ and for the corresponding value of the exponential base in
\eqref{eq:bw-bound} \cite{BandeiraRandomstrasse2025,BandeiraRandomstrasse2026}. Thus the real critical
case below is not just an example of Theorem \ref{thm:general}; it is the stated Gaussian Balan--Wang
open problem.

For full-spark $A\in\R^{(2m-1)\times m}$, \eqref{eq:omega} reduces to the minimum over square row
submatrices:
\begin{equation}\label{eq:omega-square}
  \omega(A)=\min_{T\subset[2m-1],\ |T|=m}\sigma_{\min}(A_T).
\end{equation}
Indeed, the rank-deficient sets $J$ are exactly those with $|J|\le m-1$, and adding rows can only
increase the Gram matrix. Applying Theorem \ref{thm:general} with $\F=\R$ and $N_m=2m-1$ gives the
following answer.

\begin{corollary}[Gaussian Balan--Wang base]\label{cor:bw}
Let $A_m\in\R^{(2m-1)\times m}$ have independent standard Gaussian entries. Then
\begin{equation}\label{eq:omega-limit}
  \frac{1}{m}\log\omega(A_m)\xrightarrow{P}-\log4,
\end{equation}
with the rate
\begin{equation}\label{eq:omega-rate}
  \PP\left\{\left|\frac{1}{m}\log\omega(A_m)+\log4\right|>\varepsilon\right\}=O_\varepsilon(m^{-1})
\end{equation}
for every fixed $\varepsilon>0$. Equivalently, $\omega(A_m)=4^{-m+o_P(m)}$.

Let $R_m=\max_{1\le i\le 2m-1}\|a_i\|$. For every $b>1/4$, $\PP\{\omega(A_m)\le R_m b^m\}\to1$. For
every $b<1/4$ and every fixed $C<\infty$, $\PP\{\omega(A_m)\le C R_m b^m\}\to0$. Thus the Gaussian
critical base in the Balan--Wang question is $1/4$.
\end{corollary}

The complex part of Theorem \ref{thm:general} is not a statement about the complex phase retrieval
threshold. It is the corresponding square-submatrix random matrix theorem over $\C$. For example, when
$N_m/m\to2$ one obtains $M_m^\C(A_m)=2^{-m+o_P(m)}$, reflecting the two-real-dimensional small-ball
exponent of a complex Gaussian scalar.

\section{One-matrix lower tails and binomial growth}

We first record the elementary one-matrix estimate needed for the lower bound. It is deliberately
weaker than the sharp hard-edge asymptotics of Edelman \cite{Edelman1988} and subsequent work such as
\cite{RudelsonVershynin2009,TaoVu2010}; the polynomial loss is harmless at the exponential scale
considered here.

\begin{lemma}[Polynomial least-singular-value tail]\label{lem:lsv-tail}
Let $G\in\F^{m\times m}$ have independent standard Gaussian entries, where $\F\in\{\R,\C\}$ and
$d_\F=\dim_\R\F$. There is a constant $C_\F$ such that, for all $m\ge1$ and all $t>0$,
\begin{equation}\label{eq:lsv-tail}
  \PP\{\sigma_{\min}(G)\le t\}\le C_\F m^{1+d_\F/2}t^{d_\F}.
\end{equation}
\end{lemma}

\begin{proof}
Let $R_1,\ldots,R_m$ be the rows of $G$, let $H_i=\operatorname{span}_\F\{R_j:j\ne i\}$, and set
$d_i=\dist(R_i,H_i)$. Almost surely $G$ is invertible. The negative second moment identity gives
\begin{equation}\label{eq:nsm}
  \sum_{k=1}^m \sigma_k(G)^{-2}=\sum_{i=1}^m d_i^{-2}.
\end{equation}
For completeness, if $y_i=G^{-1}e_i$, then $y_i\perp H_i$ and $\langle R_i,y_i\rangle=1$, hence
$\|y_i\|=d_i^{-1}$. Summing $\|y_i\|^2=e_i^*(GG^*)^{-1}e_i$ over $i$ proves \eqref{eq:nsm}.

If $\sigma_{\min}(G)\le t$, then the left side of \eqref{eq:nsm} is at least $t^{-2}$, so some $i$
satisfies $d_i\le \sqrt m\,t$. Conditional on the other rows, $d_i$ has the same distribution as
$|g_\F|$, where $g_\R\sim N(0,1)$ and $g_\C$ is a standard complex Gaussian; hence
$\PP\{|g_\R|\le u\}\le \sqrt{2/\pi}\,u$ and $\PP\{|g_\C|\le u\}=1-e^{-u^2}\le u^2$. A union bound over
$i$ gives \eqref{eq:lsv-tail}.
\end{proof}

We also use the following elementary asymptotic form of Stirling's formula.

\begin{lemma}[Binomial scale]\label{lem:binomial}
Let $N_m/m\to\gamma>1$. Then $\frac1m\log \binom{N_m}{m}\to h(\gamma)$, with $h$ as in
\eqref{eq:hgamma}. If $N_m=\gamma m+O(1)$, then
$\binom{N_m-1}{m-1}=\Theta_\gamma(m^{-1/2}e^{h(\gamma)m})$.
\end{lemma}

\section{Lower bound: no square submatrix is much smaller}

\begin{proposition}[Lower estimate]\label{prop:lower}
Under the assumptions of Theorem \ref{thm:general}, for every fixed $\varepsilon>0$,
\begin{equation}\label{eq:lower-prob}
  \PP\left\{M_m^\F(A_m)\le \exp\left[-\left(\frac{h(\gamma)}{d_\F}+\varepsilon\right)m\right]\right\}\to0.
\end{equation}
If $N_m=\gamma m+O(1)$, the probability in \eqref{eq:lower-prob} is bounded by
$C_{\gamma,\varepsilon,\F}m^{1+d_\F/2}e^{-d_\F\varepsilon m}$ for all large $m$.
\end{proposition}

\begin{proof}
Let $t>0$. By a union bound and Lemma \ref{lem:lsv-tail},
$\PP\{M_m^\F(A_m)\le t\}\le \binom{N_m}{m} C_\F m^{1+d_\F/2}t^{d_\F}$. Take
$t=\exp[-(h(\gamma)/d_\F+\varepsilon)m]$. Since $m^{-1}\log \binom{N_m}{m}\to h(\gamma)$, the bound
tends to zero. The stated quantitative bound follows from Lemma \ref{lem:binomial} when
$N_m=\gamma m+O(1)$.
\end{proof}

\section{Upper bound: one row and many random hyperplanes}

The matching upper estimate is the main point. We fix one row and compare it to all hyperplanes
generated by $m-1$ of the remaining rows. Let $A_m$ have rows $a_1,\ldots,a_N$, $N=N_m$. Put
$x=a_N$ and $b_i=a_i$ for $1\le i\le N-1$, and define
\[
  \cS=\{S\subset[N-1]: |S|=m-1\},\qquad L=|\cS|=\binom{N-1}{m-1}.
\]
For $S\in\cS$, let $H_S=\operatorname{span}_\F\{b_i:i\in S\}$, a.s.\ a hyperplane in $\F^m$. Choose a
unit normal $u_S$ to $H_S$ (over $\C$ determined up to a phase; all events below are phase-invariant).
For $t>0$ set
\begin{equation}\label{eq:Z-def}
  X_S=\mathbf1_{\{|\langle x,u_S\rangle|\le t\}},
  \qquad Z_t=\sum_{S\in\cS}X_S.
\end{equation}
If $Z_t>0$, then for some $S$ the row $x$ lies within distance $t$ of $H_S$; for $T=S\cup\{N\}$ and a
unit vector $v\perp H_S$, $\|A_{m,T}v\|=|\langle x,v\rangle|=\dist(x,H_S)\le t$. Thus
\begin{equation}\label{eq:Z-implies}
  Z_t>0\quad\Longrightarrow\quad M_m^\F(A_m)\le t.
\end{equation}

\subsection{Pair geometry}

\begin{lemma}[Normals for two overlapping row sets]\label{lem:normals}
Fix distinct $S,T\in\cS$. Let $r=|S\cap T|$ and $D=m-r$. Then $D\ge2$. Conditional on the rows indexed
by $S\cap T$, the two normal lines $\F u_S$ and $\F u_T$ are independent Haar-distributed lines in the
$D$-dimensional space $E^\perp$, $E=\operatorname{span}_\F\{b_i:i\in S\cap T\}$. Equivalently, for all
phase-invariant events one may take $u_S,u_T$ to be independent Haar unit vectors in $E^\perp$.
\end{lemma}

\begin{proof}
Almost surely $\dim_\F E=r$, so $\dim_\F E^\perp=D$. The sets $S\setminus T$ and $T\setminus S$ each
contain $D-1$ rows, and the corresponding row families are independent conditional on the rows in
$S\cap T$. Projecting the rows in $S\setminus T$ to $E^\perp$ gives $D-1$ independent standard Gaussian
vectors in $E^\perp$; their span is a hyperplane, and by rotational invariance its normal line is
Haar-distributed. The same argument applies to $T\setminus S$, independently.
\end{proof}

For $D\ge2$ define, with $u,v$ independent Haar unit vectors in $\F^D$ and $g$ an independent standard
Gaussian vector in $\F^D$,
\begin{equation}\label{eq:qD}
  q_D^\F(t)=\PP\{|\langle g,u\rangle|\le t,\ |\langle g,v\rangle|\le t\},
  \qquad
  p_\F(t)=\PP\{|g_\F|\le t\},
\end{equation}
where $g_\F$ is a standard scalar Gaussian over $\F$.

\begin{lemma}[Two-slab and two-disk estimate]\label{lem:two-slab}
For each $\F\in\{\R,\C\}$ there is a constant $C_\F$ such that, for $0<t\le1$,
\begin{equation}\label{eq:two-slab-main}
  q_D^\F(t)\le (1+C_\F t^2)p_\F(t)^2\left(1+\frac{C_\F}{D}\right),\qquad D\ge3,
\end{equation}
and $q_2^\F(t)\le p_\F(t)$.
\end{lemma}

\begin{proof}
The case $D=2$ is immediate, since the joint event is contained in $\{|\langle g,u\rangle|\le t\}$.
Assume $D\ge3$. For $\F=\R$, conditional on $u,v$ write $\rho=\langle u,v\rangle$ and
$s=(1-\rho^2)^{1/2}$; the pair $(\langle g,u\rangle,\langle g,v\rangle)$ is a centered real Gaussian
vector with covariance determinant $s^2$, density at the origin $(2\pi s)^{-1}$, so
\begin{equation}\label{eq:real-cond}
  \PP\{|\langle g,u\rangle|\le t, |\langle g,v\rangle|\le t\mid u,v\}\le \frac{2t^2}{\pi s}.
\end{equation}
For independent Haar unit vectors in $\R^D$,
$\EE(1-\rho^2)^{-1/2}=\Gamma(D/2)\Gamma((D-2)/2)/\Gamma((D-1)/2)^2\le1+C/D$ by the gamma-ratio
asymptotics. Averaging \eqref{eq:real-cond} and using $p_\R(t)=\sqrt{2/\pi}\,t+O(t^3)$ gives
\eqref{eq:two-slab-main} over $\R$. For $\F=\C$, conditional on $u,v$ write $s^2=1-|\rho|^2$ with
$\rho=\langle u,v\rangle$; the pair is a centered complex Gaussian vector in $\C^2$ with covariance
determinant $s^2$ and density at the origin $(\pi^2 s^2)^{-1}$, so the probability of the product of
the two disks of radius $t$ is at most $t^4/(1-|\rho|^2)$. For independent Haar unit vectors in $\C^D$,
$|\rho|^2$ has the $\mathrm{Beta}(1,D-1)$ distribution, hence
$\EE(1-|\rho|^2)^{-1}=(D-1)/(D-2)\le1+C/D$ for $D\ge3$. Since $p_\C(t)=1-e^{-t^2}=t^2+O(t^4)$, averaging
gives \eqref{eq:two-slab-main} over $\C$.
\end{proof}

\begin{lemma}[Overlap combinatorics]\label{lem:overlap}
Let $S,T$ be independent uniform elements of $\cS=\{S\subset[N-1]: |S|=m-1\}$, and set $D=m-|S\cap T|$.
Then for $d\ge1$,
\begin{equation}\label{eq:D-distribution}
  \PP\{D=d\}=\frac{\binom{m-1}{d-1}\binom{N-m}{d-1}}{\binom{N-1}{m-1}},
\end{equation}
with the convention that binomial coefficients outside their natural range vanish. Moreover
\begin{equation}\label{eq:EDinv}
  \EE\frac1D=\frac{N}{m(N-m+1)},
  \qquad
  \PP\{D=2\}=\frac{(m-1)(N-m)}{\binom{N-1}{m-1}}.
\end{equation}
\end{lemma}

\begin{proof}
Fix $S$. To have $D=d$, the set $T$ omits $d-1$ elements of $S$ and includes $d-1$ elements from
$[N-1]\setminus S$ (of size $N-m$); this gives \eqref{eq:D-distribution}. With $n=N-1$ and $k=m-1$,
\[
  \EE\frac1D=\frac1{\binom nk}\sum_{j=0}^k\frac1{j+1}\binom kj\binom{n-k}{j}
  =\frac1{\binom nk}\cdot\frac1{k+1}\sum_{j=0}^k\binom{k+1}{j+1}\binom{n-k}{j}
  =\frac{1}{\binom nk}\cdot\frac{1}{k+1}\binom{n+1}{k},
\]
using $(j+1)^{-1}\binom kj=(k+1)^{-1}\binom{k+1}{j+1}$ and Vandermonde's identity. Hence
$\EE[1/D]=\frac{n+1}{(k+1)(n+1-k)}=\frac{N}{m(N-m+1)}$. The value of $\PP\{D=2\}$ is
\eqref{eq:D-distribution} at $d=2$.
\end{proof}

\subsection{The second moment}

\begin{proposition}[Upper estimate]\label{prop:upper}
Under the assumptions of Theorem \ref{thm:general}, for every fixed $\varepsilon>0$,
\begin{equation}\label{eq:upper-prob}
  \PP\left\{M_m^\F(A_m)>\exp\left[-\left(\frac{h(\gamma)}{d_\F}-\varepsilon\right)m\right]\right\}\to0.
\end{equation}
If $N_m=\gamma m+O(1)$, the probability in \eqref{eq:upper-prob} is $O_{\gamma,\varepsilon,\F}(m^{-1})$.
\end{proposition}

\begin{proof}
Take $0<\varepsilon<h(\gamma)/d_\F$. First assume $N_m=\gamma m+O(1)$, and write $h=h(\gamma)$,
$d=d_\F$, $t=\exp[-(h/d-\varepsilon)m]\in(0,1]$ for large $m$. With $Z_t$ as in \eqref{eq:Z-def} and
$L=\binom{N-1}{m-1}$, $\mu:=\EE Z_t=L\,p_\F(t)$. Since $p_\F(t)\ge c_\F t^d$ for $0<t\le1$ and
$L\ge c_\gamma m^{-1/2}e^{hm}$ (Lemma \ref{lem:binomial}),
\begin{equation}\label{eq:mu-lower}
  \mu\ge c_{\gamma,\F}\,m^{-1/2}e^{d\varepsilon m}\to\infty.
\end{equation}
For $S\ne T$, Lemma \ref{lem:normals} and the independence of $x$ from all $b_i$ give
$\EE[X_SX_T]=q_{D(S,T)}^\F(t)$ with $D(S,T)=m-|S\cap T|$ (projecting $x$ to $E^\perp$ makes it a
standard Gaussian vector there, independent of $u_S,u_T$). Hence, for $S,T$ independent uniform in
$\cS$ and $D=m-|S\cap T|$,
\begin{equation}\label{eq:second-moment}
  \frac{\EE Z_t^2}{\mu^2}
  \le \frac{1}{L\,p_\F(t)}+
       \EE\left[\frac{q_D^\F(t)}{p_\F(t)^2}\mathbf1_{\{D\ge2\}}\right].
\end{equation}
By Lemmas \ref{lem:two-slab} and \ref{lem:overlap}, the $D\ge3$ part is at most
$(1+C_\F t^2)\bigl(1+C_\F\,\EE[1/D]\bigr)\le 1+C_\F t^2+C_{\gamma,\F}/m$ (using
$\EE[1/D]=O(1/m)$), and the $D=2$ part is at most
$\PP\{D=2\}/p_\F(t)\le C_{\gamma,\F}m^{5/2}e^{-d\varepsilon m}$. With \eqref{eq:mu-lower},
\[
  \frac{\operatorname{Var} Z_t}{\mu^2}=\frac{\EE Z_t^2}{\mu^2}-1
  \le \frac{C_{\gamma,\F}}{m}+C_{\gamma,\varepsilon,\F}m^{5/2}e^{-d\varepsilon m}
  =O_{\gamma,\varepsilon,\F}(m^{-1}),
\]
so by Chebyshev $\PP\{Z_t=0\}\le \operatorname{Var}(Z_t)/\mu^2=O(m^{-1})$. With \eqref{eq:Z-implies}
this proves \eqref{eq:upper-prob} and its rate when $N_m=\gamma m+O(1)$. For a general sequence
$N_m/m\to\gamma$, repeat the proof with $h(N_m/m)$ in place of $h(\gamma)$ and use
$h(N_m/m)\to h(\gamma)$.
\end{proof}

\section{Proofs of the main theorem and the phase retrieval corollary}

\begin{proof}[Proof of Theorem \ref{thm:general}]
Propositions \ref{prop:lower} and \ref{prop:upper} give, for every $\varepsilon>0$, that
$\frac1m\log M_m^\F(A_m)$ is below $-h(\gamma)/d_\F-\varepsilon$ and above $-h(\gamma)/d_\F+\varepsilon$
each with probability $\to0$; this is convergence in probability. When $N_m=\gamma m+O(1)$ the lower
probability is exponentially small and the upper probability is $O(m^{-1})$, giving \eqref{eq:rate}.
\end{proof}

\begin{proof}[Proof of Corollary \ref{cor:bw}]
For Gaussian $A_m\in\R^{(2m-1)\times m}$, every $m$ rows are independent almost surely, so the
reduction \eqref{eq:omega-square} holds a.s. Theorem \ref{thm:general} with $\F=\R$, $N_m=2m-1$ and
$h(2)=\log4$ gives \eqref{eq:omega-limit}--\eqref{eq:omega-rate}. For the base in \eqref{eq:bw-bound},
note $\frac1m\log R_m\xrightarrow{P}0$: the upper tail follows from a union bound and
$\PP\{\|g\|\ge e^{\eta m}\}\le \exp(-c e^{2\eta m})$ for fixed $\eta>0$, and $R_m\ge1$ with high
probability. If $b>1/4$, pick $\varepsilon$ with $e^{-(\log4-\varepsilon)}<b$; then w.h.p.\
$\omega(A_m)\le e^{-(\log4-\varepsilon)m}\le b^m\le R_m b^m$. If $b<1/4$ and $C<\infty$, pick
$\varepsilon,\eta$ with $e^\eta b<e^{-(\log4+\varepsilon)}$; then w.h.p.\
$\omega(A_m)\ge e^{-(\log4+\varepsilon)m}$ and $R_m\le e^{\eta m}$, so $CR_m b^m\le C(e^\eta b)^m
<e^{-(\log4+\varepsilon)m}$ for large $m$, whence $\PP\{\omega(A_m)\le CR_m b^m\}\to0$.
\end{proof}

\begin{remark}[What drives the constants]
The theorem separates the two inputs that determine the exponential base. The combinatorial input is
the entropy $\binom{N_m}{m}=\exp\{h(\gamma)m+o(m)\}$. The small-ball input is the real dimension of the
scalar field: $\PP\{|g_\R|\le t\}\asymp t$ while $\PP\{|g_\C|\le t\}\asymp t^2$. Thus the same binomial
entropy is divided by $d_\F=1$ over $\R$ and by $d_\F=2$ over $\C$.
\end{remark}

\section{Universality and the necessity of continuity}\label{sec:universal}

The lower estimate is insensitive to the entry law, provided it has a bounded density.

\begin{proposition}[Universal lower estimate]\label{prop:universal}
Let $A_m\in\R^{N_m\times m}$ have i.i.d.\ entries with density bounded by $K$, and $N_m/m\to\gamma>1$.
Then for every fixed $\varepsilon>0$, $\PP\{M_m^\R(A_m)\le e^{-(h(\gamma)+\varepsilon)m}\}\to0$.
\end{proposition}

\begin{proof}
The identity \eqref{eq:nsm} is purely algebraic. For $G\in\R^{m\times m}$ with rows $R_i$ and $n_i$ the
unit normal to $\operatorname{span}\{R_j:j\ne i\}$, the variable $\langle R_i,n_i\rangle=\sum_j (R_i)_j
(n_i)_j$ is, conditionally on the other rows, a sum of independent terms; the term with
$|(n_i)_{j^\ast}|=\max_j|(n_i)_j|\ge1/\sqrt m$ has density at most $K/|(n_i)_{j^\ast}|\le K\sqrt m$, so
the convolution has density at most $K\sqrt m$ and $\PP\{d_i\le s\}\le 2K\sqrt m\,s$. Hence
$\PP\{\sigma_{\min}(G)\le t\}\le 2Km^2 t$, and a union bound over $\binom{N_m}{m}=e^{h(\gamma)m+o(m)}$
blocks tends to $0$ at $t=e^{-(h(\gamma)+\varepsilon)m}$.
\end{proof}

\begin{remark}[Continuity is necessary]\label{rem:bernoulli}
For atomic entries the lower estimate fails. If the entries are $\pm1$ Bernoulli and $N=2m-1$, every
$m\times m$ submatrix is singular with probability at least $\PP\{R_1=R_2\}=2^{-m}$, so the expected
number of singular $m\times m$ submatrices is at least $\binom{2m-1}{m}2^{-m}\ge c\,2^m/\sqrt m\to\infty$.
Thus the first-moment heuristic behind Proposition~\ref{prop:universal} breaks down, and a
bounded-density assumption (in particular Gaussianity) cannot be removed.
\end{remark}

\begin{remark}[Towards a universal upper estimate]\label{rem:univ-upper}
The upper estimate is expected to persist for i.i.d.\ subgaussian entries with bounded density. The
only use of Gaussianity is Lemma \ref{lem:normals}, where rotational invariance makes the hyperplane
normals Haar-distributed. For bounded-density entries the normals are delocalized with high
probability, so a density-level local central limit theorem yields $\PP\{|\langle x,u_S\rangle|\le
t\}=\Theta(t)$ together with the matching pair estimate; making this uniform over the $\binom{N-1}{m-1}$
hyperplanes requires a delocalization bound we do not pursue here. A Berry--Esseen estimate does not
suffice at the scale $t=e^{-\Theta(m)}$, since its $O(m^{-1/2})$ error dwarfs the signal; a
density-level statement is needed.
\end{remark}


\section*{Funding}
The research presented in this paper was supported by the European Research Council (ERC) under the
European Union's Horizon 2022 research and innovation programme (grant agreement No.~101041711), by the
Simons Foundation, by Heights Labs, by the Israel Science Foundation (grant number 2258/19) and by the Israel Science
Foundation (ISF Grant 4101/25).

\section*{Declaration of generative AI and AI-assisted technologies in the manuscript preparation process}
During the preparation of this work the author used generative AI tools (large language model
assistants) to assist with drafting, revision and bibliographic organization. After using these tools,
the author reviewed and edited the content as needed and takes full responsibility for the content of
the manuscript.

\section*{Declaration of competing interest}
The author declares no competing interests.

\end{document}